\documentclass{amsproc}

\newtheorem{theorem}{Theorem}[section]
\newtheorem{lemma}[theorem]{Lemma}
\newtheorem{proposition}[theorem]{Proposition}

\newtheorem{corollary}[theorem]{Corollary}
\theoremstyle{definition}
\newtheorem{definition}[theorem]{Definition}
\newtheorem{example}[theorem]{Example}

\theoremstyle{remark}

\numberwithin{equation}{section}

\begin{document}

\title{GEOMETRIC EQUIVALENCE of $\pi$-TORSION-FREE NILPOTENT GROUPS.}

\author{R. Lipyanski}
\address{Department of Mathematics, Ben Gurion University, Beer Sheva 84105, Israel}

\curraddr{Department of Mathematics, Ben Gurion University, Beer Sheva 84105, Israel}
\email{lipyansk@math.bgu.ac.il}


\date{January 1, 1994 and, in revised form, June 22, 1994.}

\dedicatory{This article is dedicated to my teacher Prof. B. Plotkin on his 90th anniversary.}


\begin{abstract}
 In this paper we study the property of geometric equivalence of groups introduced by B. Plotkin \cite{P1, P2}. Sufficient and necessary conditions are presented for $\pi$-torsion-free nilpotent group to be geometrically equivalent to its $\pi$-completion.
 We prove that a relatively free nilpotent $\pi$-torsion-free group and its $\pi$-completion define the same quasi-variety. Examples of $\pi$-torsion-free nilpotent groups that are geometrically equivalent to their $\pi$-completions are given.
\end{abstract}

\maketitle

\section{Introduction }
 I attended Professor Plotkin's lectures about 45 years ago as a student at the University of Latvia. Over the years, I participated in Plotkin's seminars on group theory in Latvia and Israel. I am grateful to Boris Isaakovich Plotkin for  revealing
 to me the world of group theory in all its diversity. The problems posed by Plotkin at seminars and in his numerous papers have always been interesting to me. I would like to wish my teacher further creative success in Algebra in the next 30 years.

 The foundation of universal algebraic geometry for different algebraic systems were laid by B. Plotkin \cite{P1,P2, P3}. Within the framework of universal algebraic geometry (UAG), the important notion of geometric equivalence of universal algebras was introduced in \cite{P1, P2}. Geometric equivalence of two algebras means, in some sense, equal possibilities for solving systems of equations in these algebras.

  Algebraic geometry over groups was considered together with B.Plotkin by G. Baumslag, A. Miasnikov, V. Remeslennikov \cite{B1, B2} and others.

  In \cite{P3} B. Plotkin posed a question: Under what conditions is a nilpotent torsion-free group geometrically equivalent to its Mal'tsev's completion? Sufficient conditions for this equivalence were presented in \cite{T}. A criterion for a nilpotent torsion-free group to be geometrically equivalent to its Mal'tsev's completion was given by V. Bludov and  B. Guskov \cite{BG}. Based on this criterion, they also provided an example of a torsion-free three-step nilpotent group that is not geometrically equivalent to its Mal'tsev's completion.

 There exists a natural generalization  of the concept of a torsion-free group (see \cite{L1}). Let $\pi$ be a (non empty) set of primes. An integer is said to be a $\pi$-number if it is a product of powers of primes from $\pi$. A group $G$ is $\pi$-divisible if any equation of the form $x^n=a$, where $a\in G$ and $n$ is a $\pi$-number, has a solution in $G$. Group $G$ is $\pi$-torsion-free if any equation of the form  $x^n=1$, where $n$ is a $\pi$-number, has only the trivial solution in $G$. By definition, a $\pi$-divisible $\pi$-torsion-free group is called $\pi$-complete. It is known that every $\pi$-torsion-free nilpotent group $G$ can be embedded as a subgroup into a $\pi$-complete nilpotent group (see \cite{Kh1}). The minimal $\pi$-complete nilpotent group containing the group $G$ is called the $\pi$-completion of $G$. If $\pi$ consists of all primes, the $\pi$-completion of group $G$ coincides with its Mal'tsev's completion.

Here we study the problem related to geometric equivalence of a nilpotent $ \pi$-torsion-free group to its $\pi$-completion. We give necessary and sufficient conditions for this geometric equivalence. The proof of this criterion uses the theory of $\pi$-isolators of groups, as well as the approach proposed by V. Bludov and B. Guskov, to study the problem of geometric equivalence of a torsion-free nilpotent group to its Mal'tsev's completion. Based on this criterion we prove that a relatively free nilpotent $\pi$-torsion-free group is geometrically equivalent to its $\pi$-completion. If $\pi$ is the set of all primes, this result has already been proven by A. Tsurkov in \cite {T}.

In \cite{BG} it was proven that every torsion-free two-step nilpotent group is geometrically equivalent to its Mal'tsev's completion. This result corresponds to the case when $\pi$ is the set of all prime numbers. If $\pi$ is a set of primes, the situation becomes more complicated and requires additional consideration. We prove that if $G$ is a finitely generated $\pi$-torsion-free two-step nilpotent group such that the group $G/I_\pi{(G^{\prime})}$ is torsion free, then $G$ is geometrically equivalent to its $\pi$-completion. This assertion implies the above result on the geometric equivalence of an arbitrary two-step nilpotent torsion-free group to its Mal'tsev's completion.

The question of the geometric equivalence of an arbitrary $\pi$-torsion-free two-step nilpotent group to its $\pi$-completion remains open.

   \section{Preliminaries}
\subsection {Geometric equivalence of groups}
Let us recall the main definitions from \cite{P1}.

Let $F=F(X)$ be a free group generated by a finite set $X$ and $G$ be a group. The set ${\rm Hom}(F,G)$ of all homomorphisms from $F$ to $G$ can be treated as an affine space. It is possible to define a Galois correspondence $'$ between subsets in ${\rm Hom}(F,G)$ and subsets in $F$. For a subset $T\subseteq F$ define $T_{G}'$, $G$-closure of $T$
$$
T_{G}'=\{\varphi\in{\rm Hom}(F,G)\;|\;T\subseteq {\rm Ker}\varphi \}
$$
 On the other hand, for any subset $A\subseteq {\rm Hom}(F,G)$ define $A_{G}^{\prime}$, $G$-closure of $A$
 $$
 A_{G}^{\prime}=\bigcap\limits_{\mu\in A} Ker\mu
 $$
 The set $T\subseteq F$ can be treated as a system of equations that we solve in $G$. Then $T_{G}'$ is the set of all solutions in $G$ of this system. Similarly, the set $A$ can be treated as an affine space whose points are homomorphisms. Then $A_{G}'$ is the system of all equations written in $F$ such that $A$ is the set of all solutions of this system.
 \begin{definition}[\cite{P1},\cite{P4}]
 A subset $T\subseteq F$ (a subset $A\subseteq{\rm Hom}(F,G)$) is called $G$-closed if $T_G=T_{G}''$ (respectively, $A_G=A_{G}''$).
\end{definition}
\begin{definition}[\cite{P1},\cite{P2}]\label{Plot_1}
Two groups $G_1$ and $G_2$ are called geometrically equivalent if for any free group $F=F(X)$ of finite rank and for any  subset $T\subseteq F$
$$T_{G_1}''=T_{G_2}''$$
\end{definition}
In other words, the groups $G_1$ and $G_2$ are geometrically equivalent if and only if every  $G_1$-closed subset of the group $F$ is $G_2$-closed and vice versa. Note that the collection of all $G$-closed sets of the group $F$ defines the Zariski topology on it. Therefore, the groups $G_1$ and $G_2$ are geometrically equivalent if and only if they define the same Zariski topology on $F$.
\begin{proposition}[\cite{P1}]\label{cr}
Groups $G_1$ and $G_2$ are geometrically equivalent if and only if every finitely generated subgroup of $G_1$ can be approximated by subgroups of $G_2$ and vice versa.
\end{proposition}

\begin{definition}[\cite{BG},\cite{T}]
 Let $H$ and $G$ be groups and $F=F(X)$ be any free group finitely generated by $X$. If every $H$-closed  subset of $F$ is $G$-closed, we say that the topology defined by $H$ is weaker than the topology defined by $G$ and write $H\preceq G$.
 \end{definition}

The relation $\preceq$ is a preorder on the class of groups. This preorder defines an order relation on classes of geometrically equivalent groups.
\begin{proposition}[\cite{BG},\cite{T}]\label{em}
 $H\preceq G$ if and only if any finitely generated subgroup of the group $H$ can be approximated by subgroups of $G$.
\end{proposition}
\begin{proposition}[\cite{BG},\cite{T}]\label{BG_1}
$H\preceq G$ for any subgroup $H$ of $G$.
\end{proposition}

It is clear that if for groups $G_1$ and $G_2$ the conditions $G_1\preceq G_2$ and $G_2\preceq G_1$ hold, then they are geometrically equivalent.
\subsection {$\pi$-completion of a $\pi$-torsion-free nilpotent group}
 Let $\pi$ be a (non empty) set of primes. An integer is said to be a $\pi$-number if it is a product of powers of primes from $\pi$. For a subgroup $H$ of a group $G$, the $\pi$-isolator $I_\pi{(H)}$ of $H$ in $G$ is defined as follows:
  $$I_\pi{(H)}=\{ g\in G|g^n\in H\:\mbox{for some}\: \pi-\mbox{number}\: n=n(g) \}$$

 A subgroup A is said to be $\pi$-isolated if $I_\pi{(A)}=A$.

  If $\pi$ is the set of all prime numbers, the $\pi$-isolator of $H$ in $G$ is denoted by $I(H)$ and is called the isolator of $H$ in $G$.
  \begin{proposition}[see \cite{Kh1}, \cite{LR}]\label{iso}
  If $H$ is a subgroup of a nilpotent group $G$ and $\pi$ is a set of prime numbers, then $I_\pi{(H)}$ is a subgroup.
  \end{proposition}
 A group $G$ is $\pi$-divisible if for every  $\pi$-number $n$, every element $g\in G$ has an $n$th root of $g$ in $G$, i.e., there exists an element $h\in G$ such that $h^n=g$. A group is $\pi$-torsion-free if it has no non-identity element whose order is  $\pi$-number, i.e., $I_\pi{(1)}=1$.  A $\pi$-divisible $\pi$-torsion-free group is said to be $\pi $-complete.

 Let $\mathbb{Q}_\pi$ be the ring of all rational numbers whose denominators are $\pi$-numbers and $G$ be a $\pi$-divisible $\pi$-torsion-free group. An unary operation on $G$ can be defined by taking powers in $\mathbb{Q}_\pi$ satisfying the laws $(x^r)^s=x^{rs}$ for all $r,s\in \mathbb{Q}_\pi$. The group $G$ with this additional unary operation is called a $\mathbb{Q}_\pi$-powered group.
   \begin{theorem}[\cite{Kh1}]\label{Kh1}
   (a)\!  Every $\pi$-torsion-free nilpotent group $G$ of class $c$ can be embedded as a subgroup in the $\mathbb{Q}_\pi$-powered nilpotent group $\widehat{G}^\pi$ of the same class $c$ such that $\widehat{G}^\pi=I_{\pi}(G)$, i.e., for any element $g\in \widehat{G}^\pi$ there exists a $\pi$-number $m$ such that $g^m\in G$.

   (b) \!The group $\widehat{G}^\pi$  is unique up to isomorphism; moreover every isomorphism $\varphi: G \rightarrow  G^\prime $  extends to an isomorphism of $\widehat{G}^\pi$  onto $\widehat{G'}^\pi$.
    \end{theorem}

The group $\widehat{G}^\pi$ is the minimal $\pi$-complete group containing $G$ and it is called the $\pi$-completion of $G$.

 From here on we consider $\mathbb{Q}_\pi$-powered group $\widehat{G}^\pi$ as a $\pi$-torsion-free $\pi$-divisible abstract group, i.e., as the group with respect to only the usual group operations. In the special case, where $\pi$ consists of all primes, then $\mathbb{Q}_\pi=\mathbb{Q}$, the property of being $\pi$-complete coincides with the property of being complete. In this case, the $\pi$-completion $\widehat{G}^\pi$ of the group $G$ coincides with the Mal'tsev completion $\widehat{G}$ of $G$ (see \cite{M}).
  \begin{definition}
  Let $G $  be a group and $H\leq G$  be a subgroup of $G$. Denote by $H^m$ the subgroup of $G$ generated by $m$th powers of elements from $H$.
  \end{definition}
The subgroup $H^m$ is a verbal subgroup of $G$ defined by the word $x^m$.

  We also need the following lemma, given by A. Mal'tsev \cite {M}. The initial proof of this lemma is based on the theory of Lie groups. The proof of this lemma in the framework of group theory (without Lie context) was given by A. Klyachko (Lemma 2.2, in \cite{K}).  See also the proof of Lemma 1.2 in \cite{BG}.

  \begin{lemma}[\cite{M}]\label{Mal}
Let $G$ be a nilpotent group of nilpotency class $k$. Fix a natural number $n$ and denote by $G^{n^k}$ a subgroup of $G$ generated by all elements of the form $g^{n^k}, g\in G$. Then the equation $x^n=h$ has a solution in the group $G$ for each $h\in G^{n^k}$.
\end{lemma}

\section{ $\pi$-torsion-free nilpotent groups}

\subsection{Geometric equivalence of $\pi$-torsion-free nilpotent groups to their $\pi$-completions}
 Let $\pi$ be a set of primes. Now we consider some examples of $\pi$-torsion-free nilpotent groups for which we can describe their $\pi$-completions.

  \begin{example}\label{ex_1} Let $H$ be a finitely generated $\pi$-torsion-free abelian group. We wish to describe the structure of the group $\widehat{H}^\pi$.

 It is well known (see, e.g., \cite{Kr}, Ch.6, Sec.20) that the group $H$ is the finite direct product of cyclic $\pi$-torsion-free groups, i.e., there exist integers $s,k\geq 0$ and prime numbers $p_1\leq\dots\leq p_k$ together with positive integers $\alpha_1,\dots \alpha_r$ such that
\begin{equation}\label{Ab}
H\cong \mathbb{Z}^s\times \mathbb{Z}/p_1^{\alpha_1}\mathbb{Z}\times\dots\times \mathbb{Z}/p_k^{\alpha_k}\mathbb{Z},
\end{equation}
where $p_i,\, i=1,\dots,k$, are not $\pi$-numbers.

 Denote by $\mathbb{Q}^{+}_\pi$ the additive group of the ring $\mathbb{Q}_\pi$. It is clear, that $\pi$-completion $\widehat{\mathbb{Z}}^\pi$ of $\mathbb{Z}$ is equal to $\mathbb{Q}^{+}_\pi$. If $G$ is a $p$-group, where $p\notin \pi$, then for every element $g\in G$ and $m\in\pi$ there exists an $m$th root of $g$. In fact, the mapping $g\rightarrow g^m$ is an automorphism of the cyclic subgroup $<g>$. Therefore, every group $\mathbb{Z}/p_i^{\alpha_i}\mathbb{Z}, p_i\notin\pi$, is $\pi$-complete. It is easy to check that if a group $A$ is $\pi$-torsion-free nilpotent such that $A=A_1\times A_2$, then $\widehat{A}^\pi=\widehat{A_1}^\pi\times \widehat{A_2}^\pi$. As a consequence, we obtain
 $$ \widehat{H}^\pi \cong (\mathbb{Q}^{+}_\pi)^s\times \mathbb{Z}/p_1^{\alpha_1}\mathbb{Z}\times\dots\times \mathbb{Z}/p_k^{\alpha_k}\mathbb{Z} $$
 By Proposition \ref{cr}, that $H$ is geometrically equivalent to $\widehat{H}^\pi$.
  \end{example}
  Denote by $T$ the torsion group of $H$ and by $\widehat{T}^\pi$ its $\pi$-completion. It is clear that $\widehat{T}^\pi=T$, i.e., the group  $T$ is $\pi$-complete.

 Now we wish to select another class of $\pi$-complete nilpotent groups.

\begin{example}\label{ex_2}
Let $G$ be a $\pi$-torsion-free nilpotent group such that the quotient group $G/G^{\prime}$ is finite. We will show that the group $G$ is $\pi$-complete.

Since the group $G$ is nilpotent and $G/G^{\prime}$ is finite, $G$ is finite (see \cite{W}, Sec.3, p.9). Therefore, the group $G$ is the direct product of its Sylow subgroups:
$$ G=G(p_1)\times G(p_2)\times\dots\times G(p_k).$$
Since $G(p_i)$ is a $\pi$-torsion-free nilpotent group, $p_i\notin \pi$. As a consequence, the group $G(p_i)$ is $\pi$-complete.
Hence, the group $G$, as the direct product of $\pi$-complete groups, is $\pi$-complete too.
\end{example}
 Our goal is to provide a criterion for nilpotent $\pi$-torsion-free groups to be geometrically equivalent to their $\pi$-completions. For that purpose we need the result from (\cite{LR} Ch. 2, Sec. 3) that was established by A. Mal'tsev (\cite{M}) in the case when $\pi$ is the set of all prime numbers.
\begin{lemma}\label{Rob}
Let $G$ be a finitely generated nilpotent group, $g_1,\dots,g_n$ be a set of  generators of $G$, and $H$ be a subgroup of $G$. Suppose that $g_i^{r_i}\in H$ for some positive $\pi$-numbers $r_i,\;i=1,\dots,n$. Then each element of $G$ has a positive  $\pi$-power in $H$ and $\mid G:H\mid$ is a finite $\pi$-number.
\end{lemma}

We shell need the following lemma
\begin{lemma}\label{Li}
 Let $G$ be a nilpotent $\pi$-torsion-free group and $\widehat{G}^\pi$  be its  $\pi$-completion. Then for any finitely generated
 subgroup $H$ of $\widehat{G}^\pi$ there exists a $\pi$-number $r=r(G)$ such that $H^{r}\leq G$.
 \end{lemma}
 \begin{proof}
 Let the group $H$ be generated by $h_1,\dots,h_n$. Then there are $\pi$-numbers $r_i$ such that $h_i^{r_i}\in G$. Denote by $D$ the subgroup of $H$ generated by $h_1^{r_1},\dots,h_n^{r_n}$. The group $D$ is a subnormal subgroup of $H$, i.e., there is a normal series
 $$D=D_1\leq D_2\leq\dots \leq D_{n-1}\leq D_n=H$$
 from $D$ to $H$. By Lemma \ref{Rob}, the index $\mid H:D\mid$ is a $\pi$-number. Therefore, each factor in this series has an index that is a $\pi$-number. If $r_{n-1}$ is the order of the quotient group $H/D_{n-1}$, then $H^{r_{n-1}}\leq D_{n-1}$.
  By induction, it is easy to prove that there exists a sequence of $\pi$-numbers $r_{1},\dots,r_{n-1}$ such that
  $H^{r}\leq D$, where $r=r_{1}\cdots r_{n-1}$. Since $ D $ is a subgroup of $ G $, this completes the proof.
   \end{proof}

Now we derive a criterion of geometric equivalence of a $\pi$-torsion-free nilpotent group to its $\pi$-completion.
 \begin {theorem}\label{criterion}
 Let $G$ be a  $\pi$-torsion-free nilpotent group and $\widehat{G}^\pi$  be its  $\pi$-completion. Then groups $G$ and $\widehat{G}^\pi$ are geometrically equivalent if and only if the group $G$  is geometrically equivalent to subgroups $G^m$  for every $\pi$-number $m$.
\end{theorem}
\begin{proof}
Our argument goes along the lines of Theorem 3.1 in \cite{BG} but we add new assertions related to $\pi$-torsion in nilpotent groups.

Let the groups $G$ and $\widehat{G}^\pi$ be geometrically equivalent. Let us show that then groups $G$ and $G^m$ are also geometrically equivalent for any $\pi$-number $m$. The group $G^m$ is a subgroup of $G$. By Proposition \ref{BG_1}, we have $G^m\preceq G$. Let us prove that $G\preceq G^m$. Consider a finitely generated subgroup $H=<h_1,\dots,h_s>$ of $G$. Denote by $h_i^{\frac1 m }$ the root of the equation $x^m=h_i$ in $\widehat{G}^\pi$. Denote by
  $H^{\frac1 m }=<h_1^{\frac1 m },\dots,h_s^{\frac1 m}>$ the subgroup of $\widehat{G}^\pi$ generated by
   $h_i^{\frac1 m }, i=1,\dots ,s$. It is clear that $H^{\frac1 m }\preceq\widehat{G}^\pi$. Therefore, there exists an embedding $\varphi:H^{\frac1 m }\rightarrow \prod_{i\in I}\widehat{G}_i^\pi $, where $G_i, i\in I$, are isomorphic copies of $G$.  Since the groups $G$ and $\widehat{G}^\pi$ are geometrically equivalent, there exists an embedding $\psi:H^{\frac1 m }\rightarrow \prod_{i\in I}G_{i}$. Consequently, $H$ is embedded into a cartesian product of the group $G^m$. Therefore, $H\preceq G^m$. As a consequence, we obtain $G\preceq G^m$. Therefore, $G$ and $G^m$ are geometrically equivalent.

  Conversely, let the groups $G$ and $G^m$ be geometrically equivalent for every $\pi$-number $m$. We will prove that $G$ and $\widehat{G}^\pi$ are geometrically equivalent. Since $G$ is a subgroup of $\widehat{G}^\pi$, we have $G\preceq\widehat{G}^\pi$. Let us prove that $\widehat{G}^\pi\preceq G$. Let $H$ be a subgroup of $\widehat{G}^\pi$ generated by elements $h_1,\dots,h_s$.  By Lemma \ref{Li}, there exists a $\pi$-number $r$ such that $H^{r}\leq G$. By Proposition \ref{BG_1}, $H^{r}\preceq G$. If $k$ is the nilpotency class of $G$, then $G$ is geometrically equivalent to $G^{r^{k}}$. Therefore, $H^{r}\preceq G^{r^{k}}$. Hence, there exists an embedding $\varphi: H^r\rightarrow\prod_{i\in I}{G}_i^{r^{k}}$, where $G_i, i\in I$, are isomorphic copies of the group $G$. Let $\varphi_i, i\in I$, are the coordinate functions of $\varphi$, i.e., $\varphi=(\varphi_i)_{i\in I}$.

  It is clear that $\widehat{H^r}^\pi= \widehat{H}^\pi$ and $\widehat{G^{r^k}}^\pi=\widehat{G}^\pi$. By Theorem \ref{Kh1}, the mapping $\varphi$ can be extended to the embedding $\bar\varphi: \widehat{H}^\pi\rightarrow \widehat{G}^\pi $. Then $\bar\varphi(h_j^r)=\varphi(h_j^r)=(\varphi_i(h_j^r))_{i\in I}$. By Lemma \ref{Mal}, the equation $x^r=\varphi_i(h_j{^r})$ has a solution in $G_i$. Since $G_i$ is a $\pi$-torsion-free group, $\varphi_i(h_j)\in G_i$. Therefore, the mapping $\bar\varphi$ is an embedding of $H$ into a cartesian power of the group $G$. Hence, $\widehat{G}^\pi\preceq G$. As a consequence, the groups $\widehat{G}^\pi$ and $G$ are geometrically equivalent.
\end{proof}
\begin{corollary}\label{cr_2}
Let $G$ be a  $\pi$-torsion-free nilpotent group. If  every finitely generated subgroup $H$ of $G$ is geometrically equivalent to its $\pi$-completion $\widehat{H}^\pi$, then the groups $G$ and $\widehat{G}^\pi$ are geometrically equivalent.
 \end{corollary}
\begin{proof}
The groups $H$ and $\widehat{H}^\pi$ are geometrically equivalent. Therefore, by Theorem \ref{criterion}, $H$ and $H^m$ are geometrically equivalent for every $\pi$-number $m$. As a consequence, $H\preceq H^m$. By Proposition \ref{em}, there exists an embedding $\mu:H\rightarrow \prod_{i\in I}H_i^m$, where $H_i, i\in I$, are isomorphic copies of the group $H$. The mapping $\mu$ can be extended to the embedding $\bar\mu: H\rightarrow \prod_{i\in I}G_i^m $, where $G_i, i\in I$, are isomorphic copies of the group $G$. Therefore, $G\preceq G^m$. Since $G^m\preceq G$, the groups $G$ and $G^m$ are geometrically equivalent. By Theorem \ref{criterion}, the group $G$ is geometrically equivalent to its $\pi$-completion $\widehat{G}^\pi$.
 \end{proof}
\begin{corollary}
Let $G$ be a $\pi$-torsion-free abelian group. Then $G$ is geometrically equivalent to its $\pi$-completion $\widehat{G}^\pi$.
\end{corollary}
\begin{proof}
Let $H$ be a finitely generated subgroup of $G$. According to Example \ref{ex_1}, the groups $\widehat{H}^\pi$ and $H$ are geometrically equivalent. By Corollary \ref{cr_2}, the group $G$ is geometrically equivalent to its $\pi$-completion.
\end{proof}
Now we give other examples of $\pi$-torsion-free nilpotent groups which are geometrically equivalent to their $\pi$-completions $\widehat{G}^\pi$.

In what follows, we use the following Dick's Theorem (cf. \cite{Kr}, Ch.5, Sec.18)
 \begin {theorem}
Let a group $G$ be represented as follows
$$
G=<x_i\;|\;r_j(x_{j_{1}},\dots x_{j_{s}}>=1; i, j_k\in I, j\in J>
$$
Suppose that a group $D$ contains a set $\{d_i,i\in I\}$ such that $$r_j(d_{j_{1}},\dots,d_{j_{s}})=1$$
 holds in $D$ for every $j\in J$. Then the map $x_i\rightarrow d_i, i\in I$ extends to a homomorphism of $G$ onto $D$.
 \end {theorem}
We also need the following claim
 \begin{proposition}[see \cite{Kh2}]\label{Kh2}
 Let $H$ be a $\pi$-torsion-free nilpotent group. If $x^n\cdot y^m= y^m\cdot x^n$, where $x,y \in H$ and $m,n$ are $\pi$-numbers,
  then $x\cdot y=y\cdot x$.
 \end{proposition}

Now we are ready to prove
\begin{proposition}\label{A2}
Let $G$ be a finitely generated $\pi$-torsion-free two step nilpotent group such that the group $G/I_\pi{(G^{\prime})}$ is torsion-free. Then $G$ is geometrically equivalent to its $\pi$-completion $\widehat{G}^\pi$.
\end{proposition}
\begin{proof}
First, let us show that $I_\pi{(G^{\prime})}$ is an abelian group. If  $d_1,d_2\in I_\pi{(G^{\prime})}$, then there are $\pi$-numbers
$k_1, k_2$ such that $d_1^{k_1}, d_2^{k_2}\in G^{\prime}$. Since $G$ is a two step nilpotent group,
 $d_1^{k_1}\cdot d_2^{k_2}=d_2^{k_2}\cdot d_1^{k_1}$. By Proposition \ref{Kh2}, $d_1\cdot d_2= d_2\cdot d_1$. Since $G$ is a finitely generated nilpotent group, the subgroup $I_\pi{(G^{\prime})}$ is an abelian finitely generated group.

 Let us choose a basis $\bar g_1,\dots,\bar g_s$ in the abelian group $G/I_\pi{(G^{\prime})}$. Denote by $g_1,\dots,g_s$ the preimages of these elements in the group $G$. Now choose a basis $c_1,\dots,c_k$ in the abelian group $I_\pi{(G^{\prime})}$. By definition, there are $\pi$-numbers
  $p_i, i=1,\dots k$, such that $c_i^{p_i}\in G^\prime$.

  The group $G$ is generated by elements $g_1,\dots,g_s,c_1,\dots,c_k$. Let
 \begin{equation}\label{rel_d}
 R(g_1,\dots,g_s,c_1,\dots,c_k)=1
  \end{equation}
 be a relation in the group $G$. Using the collection process in the group $G$ (see ,e.g., \cite{LR}, Ch.2, Sec.2), the relation (\ref{rel_d}) can be rewritten as follows
\begin{equation}\label{rel_a}
g^{r_1}_1\cdot\ldots\cdot g^{r_s}_s \cdot c^{t_1}_1\ldots\cdot c^{t_k}_k \cdot W_1([g_i,g_j])\cdot W_2([c_l,g_d])=1,\;r_i\geq 0,t_i\geq 0,
\end{equation}
where $W_1([h_i,h_j])$ is a word in the commutators $[h_i,h_j]$ and  $W_2([c_l,g_d])$ is a word in the commutators $[c_l,g_d]$. In the group $G/I_\pi{(G^{\prime})}$ the relation (\ref{rel_a}) has the form
$$ \bar g^{r_1}_1\cdot\ldots\cdot \bar g^{r_s}_s =1 $$
Since the group $G/I_\pi{(G^{\prime})}$ is free abelian, $r_1=\dots=r_s=0$. Therefore, the relation (\ref{rel_a}) is equivalent to
\begin{equation}\label{rel_b}
 c^{t_1}_1\ldots\cdot c^{t_k}_k \cdot W_1([g_i,g_j])\cdot W_2([c_l,g_d])=1
\end{equation}

Denote by $p=\underset{i}\max\; p_i$. Raising the relation (\ref{rel_b}) to the power $p$, we have
\begin{equation}\label{rel_e}
 c^{pt_1}_1\ldots\cdot c^{pt_k}_k \cdot W_1([g_i,g_j]^p)\cdot W_2([c_l,g_d]^p)=1
\end{equation}
Since $G$ is a $\pi$-torsion-free group, the relation (\ref{rel_e}) is equivalent to (\ref{rel_b}). Since $[c_l,g_d]^p=[c_l^{p},g_d]=1$, we obtain $W_2([c_l,g_d]^p)=1$. As a consequence, the relation (\ref{rel_e}) is equivalent to
\begin{equation}\label{rel_f}
 c^{pt_1}_1\ldots\cdot c^{pt_k}_k \cdot W_1([g_i,g_j]^p)=1
\end{equation}
Note, that elements $c^{pt_i}_i$ from (\ref{rel_f}) belong to $G^{\prime}$.

For a $\pi$-number $m$ consider the map $\varphi: G\rightarrow G^m$ given by the rules:
 $$
 \varphi(g_i)=g_i^{m},\;\;\varphi(c_i)=c_i^{m^2}.
 $$
Since $[g^m_i,g^m_j]=[g_i,g_j]^{m^2}$, the relation (\ref{rel_f}) has the following form
\begin{equation}\label{rel_c}
  c^{m^2pt_1}_1\ldots\cdot c^{m^2pt_k}_k \cdot (W_1([g_i,g_j]^p))^{m^{2}}=1
\end{equation}
Since the group $G$ is $\pi$-torsion-free, the relation (\ref{rel_c}) is equivalent to the relation (\ref{rel_f}), and, as a consequence, to the relation (\ref{rel_d}). By Dick's Theorem, the mapping $\varphi$ can be extended to a homomorphism of $G$ into $G^m$. Since $G$ is a $\pi$-torsion-free group, this homomorphism is a monomorphism $G$ into $G^m$. Therefore, $G\preceq G^m$. By Theorem \ref{criterion}, $G$ is geometrically equivalent to $\widehat{G}^\pi$.
\end{proof}
\begin{corollary}[\cite{BG}]\label{nil_2}
Let $G$ be a torsion-free two-step nilpotent group. Then $G$ is geometrically equivalent to its Mal'tsev's completion $\widehat{G}$.
\end{corollary}
\begin{proof}
Let $H$ be a finitely generated subgroup of $G$. The group $H/I(H^{\prime})$ is torsion-free. By Proposition \ref{A2}, the group $G$ is geometrically equivalent to its Mal'tsev's completion $\widehat{G}$.
\end{proof}

\begin{example}
Let  $\mathbb{H}_3(\mathbb{Z})$ be the Heisenberg group  generated by elements $x,y,z$ with the defining relations $zx=xz, yz=zy, [x,y]=z$.  Elements of $H_3(\mathbb{Z})$ are all formal expressions of the form $x^{l}y^{s}z^{n}$, where $l,m,n\in\mathbb{Z}$. The group  $H_3(\mathbb{Z})$  is torsion-free two-step nilpotent. The commutator $H_3(\mathbb{Z})^{\prime}$ is generated by the element $z$. The group $H_3(\mathbb{Z})/H_3(\mathbb{Z})^{\prime}$ is a free abelian one generated by the elements $x$ and $y$. By Corollary \ref{nil_2}, the group $\mathbb{H}_3(\mathbb{Z})$ is geometrically equivalent to its Mal'tsev's completion $\widehat{\mathbb{H}_3(\mathbb{Z})}$.
\end{example}

\begin{example}
Let  $G$ be a group generated by elements $a,b,c$ with the defining relations $ca=ac, cb=bc, [a,b]=c^2, c^4=1$.  Elements of $G$ are all formal expressions of the form $a^{l}b^{s}c^{n}$, where $l,m,\in\mathbb{Z}$ and  $n\in \mathbb{Z}_3$.
The group  $G$  is two step nilpotent $\pi$-torsion-free, where $\pi$ is the set of all prime numbers without $2$. The commutator $G^{\prime}=<c^2>$ and $I_\pi{(G^{\prime}})=<c>$. The group $G/I_\pi{(G^{\prime}})$ is a free abelian one generated by the elements $a$ and $b$. By Proposition \ref{A2}, the group $G$ is geometrically equivalent to $\widehat{G}^\pi$.
\end{example}
The following question arises: Is it true that every $\pi$-torsion-free two step nilpotent group is geometrically equivalent to its $\pi$-completion?

\subsection {Geometric equivalence of a relatively free nilpotent $\pi$-torsion-free group to its $\pi$-completion.}
It was proved by A Tsurkov in \cite {T} that a torsion-free  nilpotent relatively free group is geometrically equivalent to its Mal'tsev's completion. His proof is based on the Mal'tsev correspondence technique between nilpotent $\mathbb{Q}$-powered groups and nilpotent Lie algebras over $\mathbb{Q}$ as well as  on the Lazard correspondence between groups with a central filtration and Lie rings.

Our goal is to extend  Tsurkov's result to the class of relatively free nilpotent $\pi$-torsion-free groups. We prove that a relatively free nilpotent $\pi$-torsion-free group is geometrically equivalent to its $\pi$-completion. Note that, there is no good correspondence between nilpotent $\mathbb{Q}_\pi$-powered groups and Lie $\mathbb{Q}_\pi$-algebras (see \cite{Kh1}). For this reason, here we use the technique related to the theory of isolators in groups. In the proof of this result we also apply the criterion for a $\pi$-torsion-free nilpotent group to be geometrically equivalent to its $\pi$-completion (see Theorem \ref{criterion}).

The method of associating a Lie ring to a group was studied in detail by M. Lazard \cite{L2} (see also \cite {B}).

Let $G$ be a group. Consider the lower central series of $G$:
$$
 G\geq\gamma_1{(G)}\geq\cdots\geq\gamma_i{(G)}\geq\gamma_{i+1}{(G)}\geq\cdots
$$
Let $\pi$ be a set of primes. Denote by $G_i=I_\pi{(\gamma_i{(G))}}$ the $\pi$-isolator the group $\gamma_i{(G)}$ in $G$. Note that the group $G_i=I_\pi{(\gamma_i{(G))}}$ is a subgroup of $G$, since it coincides with the preimage of the $\pi$-isolator of the identity subgroup of the nilpotent groups $G/\gamma_i{(G)}$ (see Theorem \ref{iso}).

Therefore, we have a normal series:
\begin{equation}\label{cf}
 G\geq\ G_1\geq\cdots\geq G_i\geq G_{i+1}\geq\cdots
\end{equation}
\begin{lemma}[see \cite{Kh1}]
For any group $G$ and any natural numbers $i,j$ the following holds
 $$[G_i, G_j]\leq G_{i+j},$$
 that is, $\{G_i| i=1,\dots\}$ is a central filtration of $G$.
\end {lemma}
 Now consider the additive abelian group
$$
L_{\pi}(G)=\overset{\infty}{\underset{i=1}{\oplus}} G_i/G_{i+1}\;\; \mbox{(direct sum)}
$$
Since $\{G_i| i=1,\dots\}$ is a central filtration of $G$, then we can define the structure of the Lie ring on the additive group $L_{\pi}(G)$: for any $\bar a_i\in G_i/G_{i+1}$ and $\bar a_j\in G_j/G_{j+1}$, by definition, $[\bar a_i,\bar a_j]=  [a_i,a_j] G_{i+j+1}\in G_{i+j}/G_{i+j+1}$. Then this bracket multiplication is extended to the direct sum
$L_{\pi}(G)=\overset{\infty}{\underset{i=1}{\oplus}} G_i/G_{i+1}$ by the distributive laws. We obtain the Lie ring $L_{\pi}(G)$ of the group $G$ corresponding to the central filtration (\ref{cf}). It is clear that the additive group of the ring $L_{\pi}(G)$ is $\pi$-torsion-free.

Let $F_n=F_n(X)$ be a relatively free $n$-step nilpotent group generated by $X$. A group $G=F_n/H$, where $H$ is a $\pi$-isolated verbal subgroup of $F_n$, is relatively free $n$-step nilpotent $\pi$-torsion-free. Note that if $T$ is a verbal subgroup of $F_n$, then the $\pi$-isolator $I_{\pi}(T)$ is a  $\pi$-isolated verbal subgroup of $F_n$.

The main aim of this section is to prove the following
\begin {theorem}\label{relfr}
Let $G$ be a relatively free nilpotent $\pi$-torsion-free group. Then $G$ is geometrically equivalent to its $\pi$-completion $\widehat{G}^\pi$.
\end {theorem}
\begin{proof}
By Theorem \ref{criterion},  it suffices to prove that the group $G$ is geometrically equivalent to the subgroups $G^m$ for every $\pi$-number $m$. Let $(x_i)_{i\in I}, I\subseteq \mathbb{N}$, be free generators of the group $G$. Consider the map $\varphi: G\rightarrow G$ such that $\varphi(x_i)=x_i^{m}, i\in I$, where $m$ is a $\pi$-number. Since $G$ is a relatively free group, the map $\varphi$ can be extended to the endomorphism of $G$. This endomorphism is also denoted by $\varphi$. Let us show that ${\rm{Ker}}~\varphi=0$.

To this end we consider the Lie ring $L_{\pi}(G)=\overset{n-1}{\underset{i=1}{\oplus}} G_i/G_{i+1}$, where $n$ is the nilpotency class of $G$. Since $\varphi(\gamma_i(G))\leq \gamma_i(G)$, we have $\varphi(G_i)\leq G_i$. Therefore, the endomorphism
$\varphi:G\rightarrow G$ induces the endomorphism $\bar\varphi: L_{\pi}(G)\rightarrow L_{\pi}(G)$ of the ring $L_{\pi}(G)$ such that
 $\bar\varphi(g_iG_{i+1})=\varphi(g_i)G_{i+1}$ for $g_i\in G_i$. Let us write $\bar\varphi=(\bar\varphi_i)_{i\in J}$, where $\bar\varphi_i, i\in J=\{1,\dots,n-1\}$, are the coordinate functions of $\bar\varphi$. We first prove that ${\rm{Ker}}~\bar\varphi=0$.

Let $\bar g=(\bar g_i)_{i\in J}$, where $\bar g_i=g_iG_{i+1}$, belongs to ${\rm{Ker}}~\bar\varphi$. Since $g_i\in G_i=I_\pi{(\gamma_i{(G)}})$, there exists $\pi$-number $s_i$ such that $g_i^{s_i}\in \gamma_i{(G)}$. In $L_{\pi}(G)$ we obtain
$$
s_ig_i={\underset{k}\sum v_{i_{k}}}\mod\gamma_{i+1}{(G)},
$$
where $v_{i_{k}}=v_{i_{k}}(x_{t_1},\dots,x_{t_r})$ are simple commutators in $x_{t_1},\dots,x_{t_r}$ of the length $i$.
Without loss of generality, we assume that $s_ig_i \notin G_{i+1}$. We obtain in $L_{\pi}(G)$
$$
s_ig_i={\underset{k}\sum v_{i_{k}}}\mod G_{i+1},
$$
i.e., $s_i\bar g_i={\underset{k}\sum \bar v_{i_{k}}}$, where $\bar v_{i_{k}}=v_{i_{k}}G_{i+1}$.
Thus we have
$$
\bar\varphi_{i}(s_i\bar g_i)={\underset{k}\sum \bar v_{i_{k}}}(x^m_{t_1},\dots,x^m_{t_r})=m^i{\underset{k}\sum \bar v_{i_{k}}(x_{t_1},\dots,x_{t_r})}=m^is_i\bar g_i
$$
Therefore, $s_i\bar\varphi_{j}(\bar g_i)=s_i m^i \bar g_i$. Since $s_i$ and $m$ are $\pi$-numbers and the group $G_i/G_{i+1}$ is $\pi$-torsion-free, $\bar g_i=0$. Hence, ${\rm{Ker}}~\bar\varphi=0$.

Now we prove that ${\rm{Ker}}~\varphi=0$. Let $\varphi(h)=0$ for some $h\in G_i\setminus G_{i+1}$. Since $\varphi(h)=0$, we have
$\bar\varphi(\bar h)=0$, i.e., $h\in G_{i+1}$. This contradiction proves desired.

As a consequence, we obtain that $\varphi$ is a monomorphism $G$ into $G^m$. Therefore, $G\preceq G^m$ for every $\pi$-number $m$.
It is clear that $G^m\preceq G$. Hence, $G$ is geometrically equivalent to $G^m$ for each $\pi$-number $m$. By Theorem \ref{criterion}, the group $G$ is geometrically equivalent to its $\pi$-completion $\widehat{G}^\pi$.
\end{proof}
In the special case, that $\pi$ consists of all primes, we obtain
\begin{corollary}[\cite{T}]
A relatively free nilpotent torsion-free group is geometrically equivalent to its Mal'tsev's completion.
\end{corollary}
It is known (see \cite {P4},\cite {MR}) that two nilpotent groups of the same nilpotency class are geometrically equivalent if and only if they have the same quasiidentities. By Theorem \ref{relfr}, we obtain
\begin{corollary}
A relatively free nilpotent $\pi$-torsion-free group and its $\pi$-completion define the same quasi-variety.
\end{corollary}

   \section{Acknowledgments}
I would like to express my appreciation to Prof. B. I. Plotkin for many inspiring discussions. I am also very pleased to thank
 Prof. G. I. Mashevitzky whose suggestions helped make the paper more readable.

\bibliographystyle{amsalpha}

\begin{thebibliography}{A}
\bibitem [B] {B} Yu. A. Bahturin, \textit {Identical Relations in Lie Algebras}, Utrecht:VNU, Science Press, 1987

\bibitem [B1]{B1} G. Baumslag, A. Miasnikov, V. Remeslennikov, \textit {Algebraic geometry over groups I: Algebraic sets an ideal theory}, J. Algebra, vol. 219, 1999, pp. 16-19.

\bibitem[B2]{B2} G. Baumslag, A. Miasnikov, V. Remeslennikov, \textit {Two theorem about equationnaly Noetherian groups},
 J. Algebra, vol.194, 1994, pp. 654-664.

\bibitem[BG]{BG} V.V.  Bludov, B. V. Gusev, \textit {Geometric equivalence of groups}, Proc. Steklov Inst. Math., vol.257, 2007,
 pp. 561–582.
 \bibitem[K]{K}  A.A. Klyachko \textit {Group Theory}, Preprint, http://halgebra.math.msu.su/staff/klyachko/lect4.pdf, 2007

\bibitem[Kr]{Kr} A.G. Kurosh, \textit {The theory of groups}, Nauka, Moskow, 1952.

\bibitem[Kh1]{Kh1} E.I. Khukhro,  \textit {$p$-atomorphisms of finite $p$-groups}, LMS, Lecture Note Series, vol. 246, Cambridge, 1998.
\bibitem[Kh2]{Kh2} E.I. Khukhro,  \textit {Nilpotent groups and their automorphisms}, Berlin,;Nework: de Gruyter, 1993

\bibitem[L1]{L1} M. Lazard, \textit {Problemes d'extension concernant les N-groups; inversion dela formule de Hausdorff}, Ann.
Sci.  Comptes Rendus, Paris, 237, 1953, pp. 1377-1379.

\bibitem[L2]{L2} M. Lazard, \textit {Sur les groups nilpotents et les anneaux de Lie}, Ann. Sci. Ecole Norm. Sup., 71, 1954,
pp. 101-190.

\bibitem [M]{M} A.I. Mal'tsev, \textit {Nilpotent torsion-free groups.} (Russian), Izvestiya Akad. Nauk. SSSR. Ser. Mat. vol.13 1949, pp. 201-212.

\bibitem [MR]{MR} A. Miasnikov, V. Remeslennikov, \textit {Algebraic geometry over groups I: Algebraic geometry over group II: Logical fondation}, J. Algebra, vol. 234, 2000, pp. 225-276.

\bibitem [P1]{P1} B.I. Plotkin, \textit {Varieties of algebras and algebraic varieties. Categories of algebraic varieties},
 Siberien Adv. Math., vol. 2, 1997, pp. 64-97.
 \bibitem [P2]{P2} B.I. Plotkin, \textit {Algebraic logic, varieties of algebras and algebric variety,}  Proc. Int. Alg. Conf. St. Petersburgs, New York, London, 1998.

\bibitem[P3]{P3} B.I. Plotkin, E. Plotkin, A Tsurkov, \textit {Geometricakl equivalence of groups}, Commun. Algebra, vol.8, 1999, pp. 4015-4025
\bibitem [P4]{P4} B.I. Plotkin, \textit {Seven lectures on Universal Algebraic Geometry},  Preprint of Inst. Math., Hebrew Univ.,
2000, Jerusalem.

\bibitem[LR]{LR} John C. Lennox, Derek J. Robinson \textit {The theory of infinite soluble groups} Oxford Mathematical Monographs.
              The Clarendon Press, 2004.

\bibitem [T]{T} A. Tsurkov, \textit {Geometric equivalence of nilpotent groups}, Journal of Mathematical Sciences, vol.110.  (5) 2007 pp. 407--409.

\bibitem [W]{W} Robert B. Warfield, Jr.,  {\it Nilpotent groups}, Springer Verlag, New York, 1976.
 \end{thebibliography}

\end{document}